\documentclass{aims}
\usepackage{amsmath}
  \usepackage{paralist}
  \usepackage{graphics} 
  \usepackage{epsfig} 
\usepackage{graphicx}  \usepackage{epstopdf}
 \usepackage[colorlinks=true]{hyperref}
\hypersetup{urlcolor=blue, citecolor=red}

\def\currentvolume{ }
 \def\currentissue{ }
  \def\currentyear{2020}
   \def\currentmonth{01}
    \def\ppages{  }
     \def\DOI{}

\newtheorem{theorem}{Theorem}[section]

\newtheorem{lemma}[theorem]{Lemma}

\theoremstyle{definition}

\newtheorem{remark}{Remark}

\usepackage{comment}
\usepackage[mathlines]{lineno}

\title[Summation of Gaussian shifts as Jacobi's third Theta function] 
      {Summation of Gaussian shifts as Jacobi's third Theta function}

\author[Shengxin Zhu]{}


 \email{Shengxin.Zhu@xjtlu.edu.cn}

\thanks{This research is supported by Foundation of LCP(6142A05180501), Jiangsu Science and Technology Basic Research Program (BK20171237), Key Program Special Fund of XJTLU (KSF-E-21,KSF-P-02), Research Development Fund of XJTLU (RDF-2017-02-23), and partially supported by NSFC (No.11771002, 11571047, 11671049, 11671051, 6162003, and 11871339).
}

\thanks{$^*$ Corresponding author: Shengxin Zhu}

\begin{document}
\maketitle

\centerline{Shengxin Zhu}
\medskip
{\footnotesize
\centerline{Laboratory for Intelligent Computing and Financial Technology}
 \centerline{Department of Mathematics, Xi'an Jiaotong-Liverpool University }
   \centerline{Suzhou, 215123, P.R.China}
} 



\bigskip

 \centerline{(Communicated by the associate editor name)}

\begin{abstract}
A proper choice of parameters of the Jacobi modular identity (Jacobi Imaginary transformation)
implies that the summation of Gaussian shifts on infinity periodic grids can be represented as
the Jacobi's third Theta function. As such, connection between summation of Gaussian shifts and
the solution to a Schr\"{o}dinger equation is explicitly shown. A concise and controllable upper
bound of the saturation error for approximating constant functions with summation of Gaussian shifts
can be immediately obtained in terms of the underlying shape parameter of the Gaussian. This sheds
light on how to choose a shape parameter and provides further understanding on using Gaussians with increasingly flatness.
\end{abstract}
\section{Background}
Gaussian radial basis functions(RBFs) are widely used in machine learning and in solving partial differential equations\cite{B03,F07,W05,ZW15}. It has been for long time noticed that Gaussian interpolation does not provide uniform convergence in the \emph{stationary} interpolation setting, i.e. to choose supports of RBFs proportional to the \emph{fill distance}.  Powell took such a perspective and shew that Gaussian fails badly\cite{P90}\cite[p.185]{W05}. Ron proved that interpolation with shift of Gaussian in the \emph{nonstationary} setting enjoys spectral converge \cite{R92}. Maz\'ya and Schmidt proved that \emph{approximate approximation} with Gaussian converges fast to certain threshold and then remains constant afterwards  \cite{MS96,MS07}, this saturation error can be controlled to be negligible by adjusting the \emph{shape parameter} ($\frac{1}{d}$ in the paper). In fact this provides a theoretical basis for using Gaussian with flatness limit. This will be clear in this short paper without involved mathematics as in \cite{MS96}\cite{SZ03}\cite{Wu93}. It is clear now that Gaussian is a positive definite kernel whose associate native space is `very small', and only functions belong to the underlying native space can converge very fast \cite{SW02}. For functions do not belong to native space, saturation error exists. Buhmann has pointed out that estimate for \emph{saturation order} remained to be open problems \cite[p.236]{B03}. In particularly, Buhmann applied
the cardinal function approach to derive the local polynomial reproduction properties, he proved that it is impossible for the so called \emph{pseudoadmissible} (one can ignore the definition)
basis functions to recover a constant and Gaussian is a pseudoadmissible function \cite[Thm 23]{B90}. This implies that Gaussian can not reproduce a constant function. Boyd and Wang recently studied that saturation error for Gaussian RBFs \cite{Bo10B,BW09} and showed that RBF interpolant/approximant to the constant function, 1, on infinite periodic grids must be of the following form \cite{BW09}
\begin{linenomath}
$$
1 \approx \lambda \sum_{-\infty}^{\infty} \phi(\frac{\alpha}{h}(x-mh)),
$$
\end{linenomath}
i.e. the coefficients of the basis functions must be equal on infinity periodic grids. And recently Ye designed an optimization approach to choose the shape parameters \cite{Ye16}.

Compared with Buhmann, Boyd, Ye and Wang's results, here provides a precise formula as a direct consequence of the \emph{modular identity} \cite[eq.2.3]{C03}(see also \cite{CT05}\cite{R07}), Lemma 1 bellow) or \emph{Jacobi's imaginary transformation}\cite[p.475]{WW4}(see eq.\eqref{eq:lemma} below). A simple and controllable upper bound for the saturation error in terms of the \emph{shape parameters} is immediately obtained by using a proper form of the Jacobi Theta function, without any other sophisticated mathematics. These concise results precisely reveal why the ``flatter'' Gaussian radial basis functions result in better approximation quality and how the shape parameter control the saturation error.

What we should mention is that the close relationship between summation of Gaussian shifts and the Theta function of Jacobi type has been noticed for a long time. In \cite{B94}, when Baxter estimated the lower and upper bounds of Toeplitz matrices related to Gaussian radial basis functions (Lemma 2.7 in \cite{B94}), he in fact leveraged the Poisson summation formula and the Theta function of Jacobi type to derive a sharp lower and upper bound of the quadratic forms in the following form
\begin{linenomath}
$$
      m\sum_{ j \in \mathbb{Z}^d} a_j^2     \leq \sum_{j, k \in \mathbb{Z}^d} a_j a_k f(j-k) \leq M \sum_{j \in \mathbb{Z}^d} a_j^2
$$
\end{linenomath}
where $f(x)=e^{ -\lambda \|x \|^2}$,
\begin{linenomath}
$$m=(\frac{\pi }{\lambda})^{d/2}  \sum_{k \in \mathbb{Z}^d} e^{ - \frac{\| \pi 1 +2 \pi k \|^2 }{4 \lambda} }  \text{ and } M= (\frac{\pi}{\lambda})^{d/2} \sum_{k\in \mathbb{Z}^d} e^{- \frac{  \|\pi k\|^2}{\lambda}}.$$
\end{linenomath}
Such a technique was further developed to derive the explicit formula for the condition number of certain bi-infinite Toeplitz matrix in the form of a multiple of the Jacobian elliptic function $dn$ \cite[Thm 2.7]{BS96}\cite{RS99}. The Jacobi Theta function was also involved in analysis of the approximate approximation with Gaussian in\cite{MS96}, and recently in orthogonalization for systems consisting of Gaussian shifts \cite{ZKMS11}. These results are technically elegant. However, it seems that such a connection between Gaussian and the Jacobi Theta function has not been mentioned in popular books on RBF \cite{B03,F07,FM15,W05,SZ03} and other many publications as far as we know.  The novelty here lies on that concise constructive results are obtained only by choosing proper parameters of a modular identity of the Jacobi Theta function. This modular identity or imaginary transformation is often used to obtain definite numerical results in problems involving \emph{elliptic functions}, \emph{elliptic curves}, and \emph{modular groups} \cite{H04}. No other sophisticated harmonic analysis is used. Limitation of this method is that such a formulae is concise for one dimensional cases, for higher dimensional cases, involved mathematics as in \cite{MS96} are required.

Such a concise formulae provides an alternative perspective to using flat radial basis functions, which was largely popularized due to Driscoll and Fornberg's investigation \cite{DF02}.

\section{Main Results}

\begin{theorem}
\label{thm:1}
Let $\vartheta_3(z,\tau)$ be the third Theta function defined as
\begin{linenomath}
\begin{equation}
\vartheta_3(z,\tau) =\sum_{ n\in \mathbb{Z}} e^{2  i n z + \pi i n^2 \tau} \label{eq:def}
\end{equation}
\end{linenomath}
on $\mathbb{C} \times \mathbb{H}$, where $\mathbb{H}\subset \mathbb{C}$ is the upper half complex plane. i.e. $Im (\tau) >0$.
If $d$ is positive constant, then
\begin{linenomath}
\begin{equation}
 \frac{1}{\sqrt{\pi d}}\sum_{n=-\infty}^{\infty} e^{-\frac{(z-n\pi)^2}{d\pi^2}} = \vartheta_3(z, i \pi d). \label{eq:result}
\end{equation}
\end{linenomath}
or equivalently
\begin{linenomath}
\begin{equation}
 \frac{1}{\sqrt{\pi d}}\sum_{n=-\infty}^{\infty} e^{-\frac{(z-n)^2}{d}} = \vartheta_3(\pi z, i \pi d). \label{eq:result2}
\end{equation}
\end{linenomath}
\end{theorem}

As a direct consequence of the result, the saturation error for approximating the constant function 1 on infinity periodic grids can be controlled to be negligible by adjusting the \emph{shape parameter} $d$.
\begin{theorem}
\label{thm:2}
If $d$ is positive constant, then
\begin{linenomath}
\begin{equation}
\left| \frac{1}{\sqrt{\pi d}}\sum_{n=-\infty}^{\infty} e^{-\frac{(x-n)^2}{d}} -1 \right| = \left| \vartheta_3(\pi z, i \pi d)-1 \right| < \mathrm{csch(\pi^2 d)}
\label{eq:th2}
\end{equation}
\end{linenomath}
\end{theorem}

\begin{figure}[ht]
\begin{minipage}[t]{0.45\linewidth}
\centering
\includegraphics[width=0.95\linewidth]{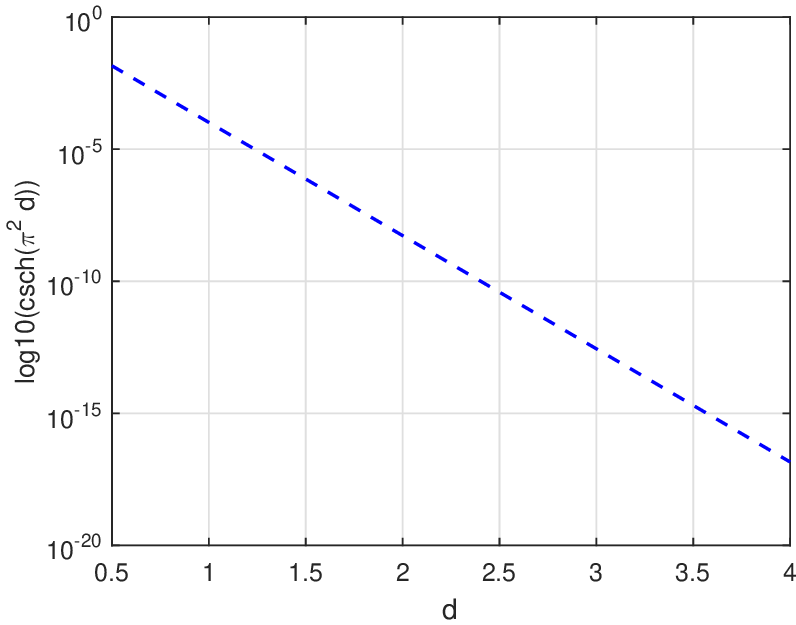}
\caption{$ \log_{10} \csc h (\pi^2 d)$}
\end{minipage}%
\begin{minipage}[t]{0.45\linewidth}
\centering
\includegraphics[width=0.95\linewidth]{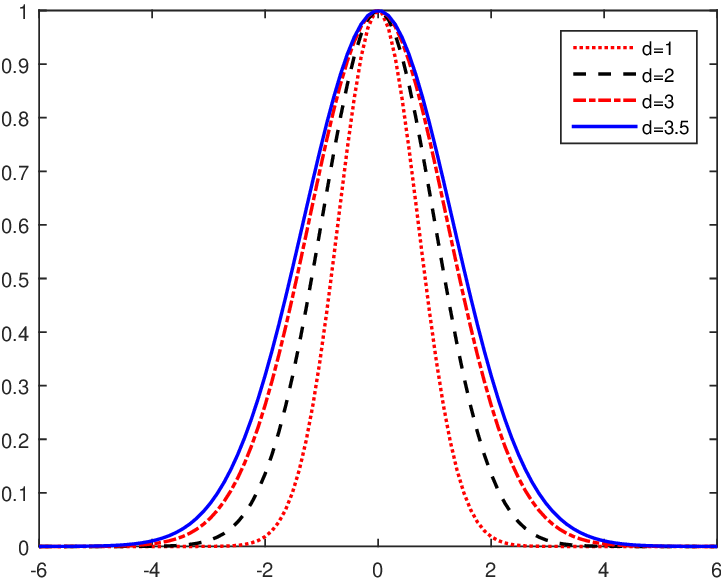}
\caption{ $e^{-\frac{x^2}{d}}$ }
\end{minipage}%
\label{figu}
\end{figure}

\begin{figure}[ht]
\begin{minipage}[t]{0.45\linewidth}
\centering
\includegraphics[width=0.9\linewidth]{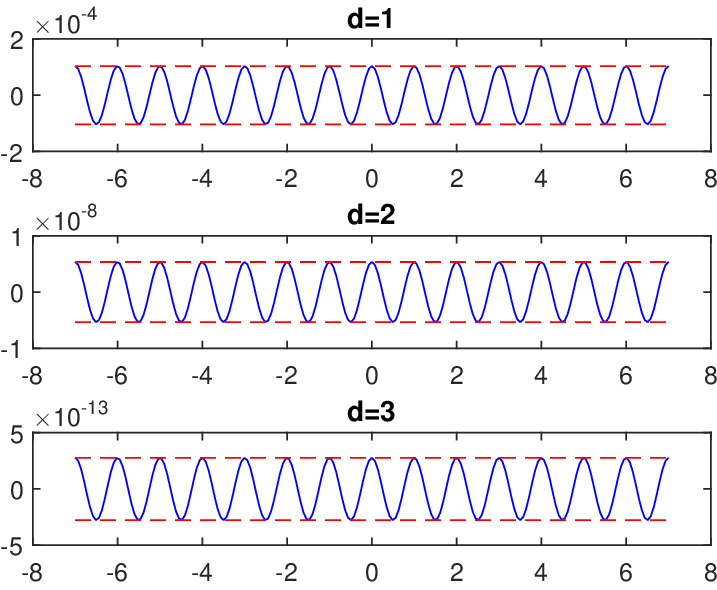}
\end{minipage}%
\begin{minipage}[t]{0.45\linewidth}
\centering
\includegraphics[width=0.9\linewidth]{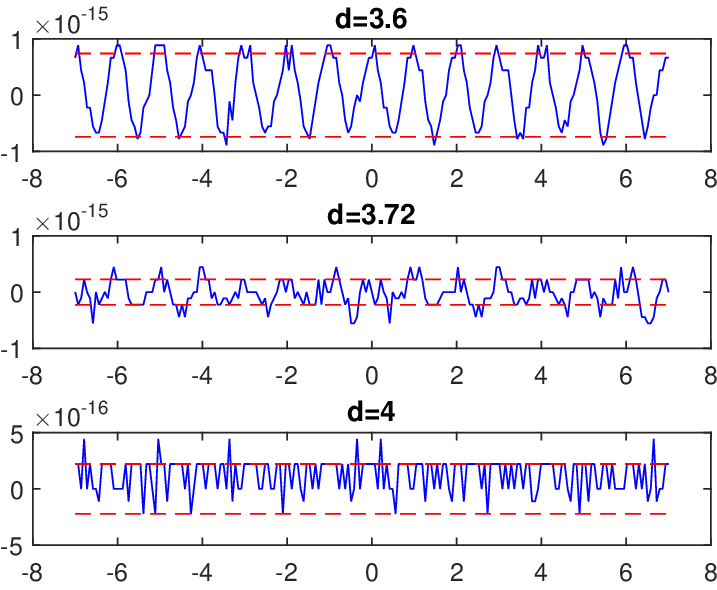}
\end{minipage}%
\caption{ error ({\color{blue}-}) and control line ({\color{red}-.-})}
\end{figure}
To demonstrate Theorem \ref{lem:2}, we plot the $\log_{10} \csc(\pi^2 d)$ against $d$ in FIGURE 1. And we use the basis functions in FIGURE 2 to approximate the constant function 1, and plot the error function
$$
error(x)= \frac{1}{\sqrt{\pi d}} \sum_{n=-1000}^{1000} e^{-\frac{(x-n)^2}{d}} -1
$$
and the control terms $ \pm \csc h(\pi^2 d)$ in FIGURE 3 (for $d=4$, the control terms are $\pm \epsilon$, the machine precision). This is a vivid example which can mathematically show that \emph{flatter} Gaussian provides better approximation.

\section{Proof the main results}

To make this communication self-contained, we first introduce some relevant results to the Theta function defined in \eqref{eq:def}, which is also called the third Jacobi Theta function and denoted as $\vartheta_3(z,\tau)$. It is obvious that $\vartheta_3(z,\tau)= \vartheta_3(-z,\tau)$ is an even function with respect to $z$ and can be written as the following equivalent form
\begin{equation}
 \vartheta_3(z,\tau) = 1+ 2\sum_{n=1}^{\infty} q^{n^2} \cos 2 \pi n  z ,
\label{lem:1}
\end{equation}
where $q= e^{ i \pi \tau}$.  This representation will be used to derive the result in Theorem \ref{thm:2}. The following \emph{modular identity} or \emph{Jacobi's imaginary transformation} plays  an essential role in deriving our main results.

\begin{lemma}
\label{lem:2}
If $\vartheta(z,\tau)$ is defined in \eqref{eq:def}, then
\begin{linenomath}
\begin{equation}
\vartheta_3(z, \frac{-1}{\tau}) =\frac{1}{\sqrt{ -i \tau}} e^{\frac{z^2}{i \pi \tau}} \vartheta_3(\frac{z}{\tau}, -\frac{1}{\tau}). \label{eq:lemma}
\end{equation}
\end{linenomath}
\end{lemma}

\begin{proof} This is the exactly formula in \cite[p.475]{WW4} and \cite[eq.4.1]{Bel61}. One can also refer to
\cite[eq.2.3]{C03}(also \cite{CT05}\cite{R07}) for a brief proof of an equivalent identity.
\end{proof}

\section*{Proof of Theorem \ref{thm:1}}

\begin{proof}

By applying the modular identity to $\vartheta_3(z, i \pi d)$, we have
 \begin{linenomath}
 \begin{align*}
  \vartheta_3(z, i\pi d) & = \frac{1}{\sqrt{\pi d}}  e^{ -\frac{z^2}{\pi^2 d} }
  \vartheta_3(\frac{z}{i\pi d},-\frac{1}{i \pi d} ) \\
 & = \frac{1}{ \sqrt{\pi d}} e^{ -\frac{z^2}{\pi^2 d} }
   \sum_{ n\in \mathbb{Z}} e^{2 n i \frac{z}{i \pi d} +  i \pi n^2 (-\frac{1}{i\pi d} ) }
    \\
   &=\frac{1}{ \sqrt{\pi d}}\sum_{ n\in \mathbb{Z}}e^{- \frac{(z -\pi n)^2}{\pi^2 d}} .
 \end{align*}
 \end{linenomath}
Similarly we can obtain \eqref{eq:result2}.
\end{proof}

Equivalent results in Theorem \ref{thm:1} have appeared elsewhere.
\begin{remark}
In the classical book \cite[p.476]{WW4}, the identity is in the form
\begin{equation}
\sum_{n \in \mathbb{Z}} e^{2 n i z + n^2 \pi i \tau} = \frac{1}{\sqrt{i \tau}} \sum_{-\infty}^{\infty} e^{\frac{(z-n\pi)^2}{\pi i \tau}}.
\label{eq:477}
\end{equation}
Set $\tau= i \pi d$ in \eqref{eq:477}, we can obtain \eqref{eq:result}.
\end{remark}
\begin{remark}
In \cite[eq.19.2]{Bel61}, the formula reads as
$$
1+2 \sum_{m=1}^{\infty} e^{-m^{2} \pi^{2} t} \cos 2 m \pi x=\frac{1}{\sqrt{\pi t}} \sum_{n=-\infty}^{\infty} e^{-(n+x)^{2} / t}.
$$
Let $t=d$, and change $n$ to $-n$ in the right hand side we can obtain \eqref{eq:result2}.
\end{remark}

\section*{Proof of Theorem \ref{thm:2}}
\begin{proof}
According to the equivalent formulation of $\vartheta_3$ in  \eqref{lem:1}, let $q= e^{i \pi \tau} = e^{-\pi^2 d}$ we have
\begin{linenomath}
\begin{align*}
& \left| \frac{1}{\sqrt{\pi d}}\sum_{n=-\infty}^{\infty} e^{-\frac{(x-n \pi)^2}{d\pi^2}} -1 \right|= \left|\vartheta_3(z, i \pi d) -1\right| =2 \left| \sum_{n=1}^{\infty} q^{n^2} \cos 2 \pi n z  \right|\\
& =2\left| \sum_{n=1}^\infty e^{-\pi^2 d n^2}  \cos  2 \pi n z \right| \leq 2\sum_{n=1}^{\infty} \left|e^{-\pi^2 d n^2}\right| \left| \cos 2 \pi n z  \right| \\
& \leq 2\sum_{n=1}^{\infty} e^{-\pi^2 d n^2}= \vartheta_3(0,i\pi d)-1 \\
& <2\sum_{n=1}^{\infty} e^{-\pi^2 d(2n-1)} = 2 e^{\pi^2 d}  \sum_{n=1}^{\infty} (e^{-2\pi^2 d})^n=e^{\pi^2 d} \frac{ 2e^{-2 \pi^2 d}}{ 1- e^{-2 \pi^2 d}}\\
&= \frac{2 e^{-\pi^2 d}}{1 - e^{-2 \pi^2 d}} = \frac{2}{ e^{\pi^2 d}- e^{-\pi^2 d}}=csch(\pi^2 d).
\end{align*}
\end{linenomath}
For the last inequality, notice that $n^2 \geq 2n -1 $.

\end{proof}

\section{Discussion}

Because $Im(\tau)>0$, then $|e^{\pi i n^2 \tau}|<1$, the infinite series in \eqref{eq:def} is absolutely convergent to $\theta(z,\tau)$, then the derivative of the summation is equal to the summation of the derivatives\cite[p.470]{WW4}:
\begin{linenomath}
\begin{align*}
\frac{d^2 \vartheta_3(z,\tau)}{dz^2}  & = \sum_{n\in \mathbb{Z}} \frac{d^2}{dz^2} e^{2 i n z + \pi i n^2 \tau} = -4 n^2 \sum_{n\in \mathbb{Z}} e^{2 i nz +\pi i n^2 \tau} \\
 & = -\frac{4}{\pi i} \sum_{n\in \mathbb{Z}} \pi i \pi n^2 e^{2 i nz +\pi i n^2 \tau}  = -\frac{4}{\pi i}\frac{d \vartheta_3(z,\tau)}{d \tau}.
\end{align*}
\end{linenomath}
Therefore, both the summation of the Gaussian shifts and the Theta function satisfy the following Schr\"{o}dinger equation. Such a connection may shed new light on investigating Gaussian RBFs.

The present results only works for the one dimensional case on infinite periodic grids. For higher dimensional cases, involved mathematics should be applied as in \cite{MS96}\cite{FHW12}. The results may be related to the quasi-interpolation as said by the author \cite{Gao12}\cite{GXWZ20} and related to the learnability \cite{Zhou07}. Finally, we would like to point out that recent research has connected a class of compact radial basis functions with the hypergeometric functions \cite{CH14,S11,H12}. It seems that there are various research opportunities in this direction.


\providecommand{\href}[2]{#2}
\providecommand{\arxiv}[1]{\href{http://arxiv.org/abs/#1}{arXiv:#1}}
\providecommand{\url}[1]{\texttt{#1}}
\providecommand{\urlprefix}{URL }


\medskip
Version 1,  xxxx 20xx; revised xxxx 20xx.
\medskip

\end{document}